\documentclass[reqno,12pt]{amsproc}

\usepackage{amsbsy}
\usepackage{amsfonts}
\usepackage{amsmath}
\usepackage{amsthm}
\usepackage{amssymb}
\usepackage{enumerate}

\usepackage{hyperref}
\usepackage{color}
\usepackage{etoolbox}
\expandafter\patchcmd\csname\string\proof\endcsname
  {\normalparindent}{0pt }{}{}
\usepackage[margin=1in]{geometry}

\begin{document}

\title[Bootstrapped zero density estimates]{Bootstrapped zero density estimates and a central limit theorem for the zeros of the zeta function}

\author{Kenneth Maples, Brad Rodgers}

\address{Institut f\"{u}r Mathematik, Universit\"{a}t Z\"{u}rich, Winterthurerstr. 190, CH-8057 Z\"{u}rich}
\email{kenneth.maples@math.uzh.ch, brad.rodgers@math.uzh.ch}

\maketitle

\theoremstyle{plain}

\newtheorem{theorem}{Theorem}[section]
\newtheorem{lemma}[theorem]{Lemma}
\newtheorem{prop}[theorem]{Proposition}
\newtheorem{corollary}[theorem]{Corollary}
\newtheorem{conj}[theorem]{Conjecture}
\newtheorem{dfn}[theorem]{Definition}

\newcommand\MOD{\textrm{ (mod }}
\newcommand\Var{\mathrm{Var}}
\newcommand\spann{\mathrm{span}}
\newcommand\E{\mathop{\vcenter{\hbox{\relsize{+2}$\mathbf{E}$}}}}
\newcommand\Prob{\mathop{\vcenter{\hbox{\relsize{+2}$\mathbf{P}$}}}}
\newcommand\Tr{\mathrm{Tr}}
\newcommand\Ss{\mathcal{S}}
\newcommand\Dd{\mathcal{D}}
\newcommand\Uu{\mathcal{U}}
\newcommand\supp{\mathrm{supp}\;}
\newcommand\sgn{\mathrm{sgn}}
\newcommand\bb{\mathbb}
\newcommand\pvint{-\!\!\!\!\!\!\!\int_{-\infty}^\infty}
\def\res{\mathop{\mathrm{Res}}}

\begin{abstract}
We unconditionally prove a central limit theorem for linear statistics of the zeros of the Riemann zeta function with diverging variance. Previously, theorems of this sort have been proved under the assumption of the Riemann hypothesis. The result mirrors central limit theorems in random matrix theory that have been proved by Szeg\H{o}, Spohn, and Soshnikov among others, and therefore provides support for the view that the zeros of the zeta function are distributed like the eigenvalues of a random matrix. 

A key ingredient in our proof is a simple bootstrapping of classical zero density estimates of Selberg and Jutila for the zeta function, which may be of independent interest.
\end{abstract}



\section{Introduction}
\label{1}

The purpose of this note is to unconditionally prove a central limit theorem for linear statistics  of the zeros of the Riemann zeta function. To denote the non-trivial zeros, we use the notation $1/2+i\gamma$. The Riemann Hypothesis is the statement that $\gamma$ is real for all zeros, but we do not assume it here, so that in what follows $\gamma$ may be complex. To be able to refer to the coordinates of zeros more directly, we also denote non-trivial zeros by $\beta_0+i\gamma_0$, where $\beta_0, \,\gamma_0\in \mathbb{R}$. A result dating back to Riemann's memoir states that the ordinates $\gamma_0$ occur with density roughly $\log T/2\pi$ near a large height $T$. More exactly, defining
\begin{equation}
\label{N_count}
N(T):= \#\{\gamma_0\in (0,T)\},
\end{equation}
\begin{theorem}[Riemann-von Mangoldt]
\label{RvM}
$$
N(T) = \frac{T}{2\pi}\log \frac{T}{2\pi} - \frac{T}{2\pi} + O(\log T).
$$
\end{theorem}

Our main result is a more precise characterization of the distribution of zeros:

\begin{theorem}
\label{mainCLT} 
Let $n(T)$ be any fixed function tending to infinity as $T\rightarrow\infty$ in such a way that $n(T) = o(\log T)$, and let $t$ be a random variable uniformly distributed on the interval $[T,2T]$. Let $\eta$ be a fixed real valued function with compact support and bounded variation. Define the count $\Delta_\eta$ by
$$
\Delta_\eta = \Delta_\eta(t,T):= \sum_\gamma \eta\Big(\frac{\log T}{2\pi n(T)}(\gamma_0-t)\Big),
$$where the sum is over all zeros $\gamma$, counted with multiplicity. So long as $\int |x||\hat{\eta}(x)|^2\,dx$ diverges, we have the following:
\begin{equation}
\label{expectation}
\mathbb{E} \Delta_\eta = n(T)\int_\mathbb{R} \eta(y)\,dy + o(1),
\end{equation}
\begin{equation}
\label{variance}
\Var\; \Delta_\eta \sim \int_{-n(T)}^{n(T)} |x||\hat{\eta}(x)|^2\,dx,
\end{equation}
and in distribution,
\begin{equation}
\label{normal}
\frac{\Delta_\eta - \mathbb{E}\Delta_\eta}{\sqrt{\Var\;\Delta_\eta}} \Rightarrow \mathcal{N}(0,1)
\end{equation}
as $T\rightarrow\infty$, where $\mathcal{N}(0,1)$ is a standard normal random variable.
\end{theorem}

\textit{Remark:} To avoid confusion, we emphasize that $N(T)$ and $n(T)$ are different functions which play different roles in the sequel. The former is a particular function defined by the relation \eqref{N_count}, while the latter may be \emph{any} function that meets the requirements of the theorem.

Under the assumption of the Riemann Hypothesis this result was proved independently in \cite{BoKu} and \cite{Ro}. In the present paper, we remove this assumption from its proof.

It is worth noting that the condition that $\int |x||\hat{\eta}(x)|^2\,dx$ diverge should not be necessary for the conclusion the statement of of Theorem \ref{mainCLT}, and under the assumption of the Riemann Hypothesis, results of this sort are proved in both \cite{BoKu} and \cite{Ro}.

We want also to note that we have recently found out that in work forthcoming, P. Bourgade, J. Kuan, and M. Radziwi\l\l\,  have independently removed the assumption of the Riemann Hypothesis from the proof of central limit theorems of this sort, and indeed have been in possession of such results since this summer (private communication). They deserve priority for an unconditional proof for this reason. The two techniques for making the proof unconditional, however, differ considerably. In particular the bootstrapping of short interval zero density estimates we make use of here is of interest independent of Theorem \ref{mainCLT}.

\vspace{2mm}

\textit{Motivating remarks:} We note that setting $\eta = \mathbf{1}_{[0,1]}$ in Theorem \ref{mainCLT} recovers a classical central limit theorem of A. Fujii (also proved uncondtionally) \cite{Fu, Fu2}. In this case $\Delta$ is a count of the number of $\gamma_0$ that lie in the interval $[t, t + 2\pi n(T)/\log T)$:
$$
\Delta_{\mathbf{1}_{[0,1]}} = N(t + \tfrac{2\pi n(T)}{\log T}) - N(t),
$$
while
$$
\mathbb{E} \Delta_{\mathbf{1}_{[0,1]}} = n(T) + o(1)
$$
$$
\Var\; \Delta_{\mathbf{1}_{[0,1]}} \sim \frac{1}{\pi^2} \log n(T).
$$

Results such as this and Theorem \ref{mainCLT} are referred to as \emph{mesoscopic}. This means that they concern collections of consecutive zeros that are expected to contain more and more zeros as their height $T$ increases, but collections whose expected number of zeros is $o(\log T)$; note that $n(T)$ plays the role of the expected number of zeros. Statements about collections of consecutive zeros that are essentially bounded in number are known as \emph{microscopic}, while statements about collections that grow in expected number like $\log T$ or faster are known as \emph{macroscopic}\footnote{In the macroscopic setting, by the density estimate of Theorem \ref{RvM}, no ``zooming in" is required to see such a collection of zeros.}. 

That the variance grows much more slowly than the expectation in Fujii's result and Theorem \ref{mainCLT} is indicative of a rigidity in the distribution of zeros at the mesoscopic scale. Indeed, in the mesoscopic regime, zeros of the zeta function are expected to exhibit universalitity in that they statistically resemble the eigenvalues of a random matrix \cite{Be}. The more general linear statistics of Theorem \eqref{mainCLT} are a matter of long standing interest in random matrix theory, dating back to a central limit theorem proved by Szeg\H{o} for the eigenvalues of a random unitary matrix (the Strong Szeg\H{o} theorem) \cite{Sz}. In this connection, see also \cite{DiEv, So1, Sp}. Theorem \eqref{mainCLT} mirrors exactly these results, and therefore provides support for the view that mesoscopically the distribution of zeros is modeled accurately by random matrix theory. Further motivating discussion may be found in the aforementioned papers \cite{BoKu, Ro}.

\vspace{2mm}

\textbf{Notation:} We follow standard conventions of analytic number theory, so that $e(x) = e^{i2\pi x}$, the Fourier transform of a function is $\hat{f}(\xi) = \int e(-x \xi) f(x)\,dx$ and the inverse Fourier transform is $\check{g}(x) = \int e(x \xi)g(\xi)\,d\xi.$ In these formula, we allow $\xi$ and $x$ to be complex numbers, provided the integrand remains integrable. In particular, we follow the convention if $f:\mathbb{R}\rightarrow\mathbb{C}$ has a Fourier transform that is compactly supported, we may extend $f$ harmonically to a function defined on $\mathbb{C}$ with $f(x+iy) = \int e((x+iy)\xi) \hat{f}(\xi)\,d\xi.$ Convolution is denoted by $f\ast g(x) = \int f(y) g(x-y)\,dy.$ We use the notation\footnote{Because we will several times reference the argument in \cite{Ro}, we note that in that paper the symbol $\lesssim$ is used in place of the symbol $\ll$ used here.} $f(x) \ll g(x)$ and $f(x) = O(g(x))$ interchangeably to mean there is a constant $C$ not depending on $x$ so that $f(x) \leq Cg(x)$. $|f(x)| \ll_A g(x)$ and $f(x) = O_A(g(x))$ mean that the constant $C$ may depend on $A$.

\section{Main tools}
\label{2}

Our proof proceeds by modifying the argument of \cite{Ro}. We outline the main tools and necessary ideas here.

As in almost all results of this sort, we will make use of the explicit formula, due in varying stages of generality to Riemann \cite{Ri}, Guinand \cite{Gu}, and Weil \cite{We}, relating the zeros of the zeta function to the primes.

\begin{theorem}[The explicit formula]
\label{explicit}
For a continuous and compactly supported function $g$,
$$
\lim_{V\rightarrow\infty} \sum_{|\gamma|< V}\hat{g}\Big(\frac{\gamma}{2\pi}\Big) - \int_{-V}^V \hat{g}\Big(\frac{\xi}{2\pi}\Big) \frac{\Omega(\xi)}{2\pi}\,d\xi = \int_{-\infty}^\infty (g(x) + g(-x)) e^{-x/2} d\big(e^x - \psi(e^x)\big),
$$
where
$$
\psi(x):= \sum_{n\leq x} \Lambda(n),
$$
with $\Lambda$ the von Mangoldt function, and 
$$
\Omega(\xi) := \frac{1}{2}\frac{\Gamma'}{\Gamma}\Big(\frac{1}{4}+i\frac{\xi}{2}\Big) +\frac{1}{2}\frac{\Gamma'}{\Gamma}\Big(\frac{1}{4}-i\frac{\xi}{2}\Big) - \log \pi.
$$
\end{theorem}

\textit{Remark:} Note that, by Stirling's formula,
\begin{equation}
\label{Stirlings}
\frac{\Omega(\xi)}{2\pi} = \frac{\log\big((|\xi|+2)/2\pi\big)}{2\pi} + O\Big(\frac{1}{|\xi|+2}\Big),
\end{equation}
so that this term corresponds to an approximation of the density of zeros near height $\xi$. On the other hand,
$$
\int_{-\infty}^\infty g(x) e^{-x/2}d\big(e^x - \psi(e^x)\big) = \int_0^\infty \frac{g(\log t)}{\sqrt{t}}\,dt - \sum_{n=1}^\infty \frac{g(\log n)}{\sqrt{n}}\Lambda(n),
$$
with the term $\int g(\log t)/\sqrt{t}\,dt$ an approximation to $\sum g(\log n)\Lambda(n)/\sqrt{n}$.

\vspace{2mm}

A proof of the explicit formula may be found in several places, for instance \cite[pp. 410-416]{MoVa} or \cite[pp. 108-109]{IwKo}.

\vspace{2mm}

In order to cite a result of Fujii, we require the function $S(t)$ defined by the relation,
$$
S(t) = \frac{1}{\pi}\arg\zeta(1/2+it),
$$
where as long as $t$ is not the height of a zero, the argument is defined by continuity along a rectangle beginning at 2, passing to $2+it$, and then to $1/2+it$. If $t$ is the height of a zero, $S(t)$ is defined by lower semicontinuity. 

One may also equivalently define the function $S(T)$ by the relation
\begin{equation}
\label{SofT}
N(T) = \int_0^T \frac{\Omega(\xi)}{2\pi} \,d\xi + 1 + S(T),
\end{equation}
so that $S(T)$ is an error term of the zero counting function $N(T)$  \cite[pp. 452]{MoVa}.

\begin{theorem}[Fujii]
\label{Fuj1} 
For fixed $a>0$, and $T^{1/2+a} \leq H \leq T$, $0\leq h \leq  1,$
$$
\int_{T}^{T+H} (S(t+h)-S(t))^{2k}\,dt = \frac{c_{2k}}{\pi^{2k}} H \log^{2k}(2+h\log T)\Big(1+ O_k\big(\log^{-1/2}(2+ h\log T)\big)\Big),
$$
where $c_{2k}:= (2k-1)(2k-3)\cdots 3 \cdot 1$ are the moments of a standard normal random variable.
\end{theorem}

This is the Main Theorem of \cite{Fu}.\footnote{Note that in the statement of Theorem \ref{Fuj1} in \cite{Fu}, there is an error in the admissible range of $h$. This is noted and corrected in \cite{Fu3}, and further in \cite{Fu2}.} When $h\log T\rightarrow\infty$ this is a computation of the $2k^{\textrm{th}}$ moment in the aforementioned central limit theorem of Fujii. When $h\log T = O(1)$, this gaussian information is lost (and indeed a central limit theorem will not be true in this microscopic range), but even in this case, as noted by Fujii, his estimate can be used as an upper bound for the average number of zeros in a microscopic interval, by making use of \eqref{SofT}. We develop estimates of this sort suited for our purposes in section \ref{4}.

In this connection, we also note a pointwise upper bound for zeros in a macroscopic interval, which follows straightforwardly from the Riemann-von Mangoldt formula, Theorem \ref{RvM}.

\begin{corollary}
\label{ptwise}
$$
N(T+1)-N(T) \ll \log(|T|+2).
$$
\end{corollary}

\vspace{2mm}

We also will need,

\begin{theorem}[A zero density estimate]
\label{SelbergJutila}
Let $N(\sigma,T)$ be the number of non-trivial zeros with $\beta_0 > \sigma$ and $\gamma_0 \in (0,T)$. Then for any fixed constant $c\in (0,1)$,
$$
N(\sigma,T) \ll_c (T\log T) T^{-c(\sigma-1/2)}.
$$
\end{theorem}

This theorem was proved by Selberg \cite{Se2} for the constant $c = 1/4$, and subsequently improved by Jutila \cite{Ju} to the result above. For our purposes, any constant $c$, no matter how small, would be sufficient.

This zero density estimate, as written above, is clearly a global theorem. Nonetheless, just as the Riemann-von Mangoldt formula can be seen as a characterization of the average density of zeros at a local scale, so can Theorem \ref{SelbergJutila} be seen as an $L^1$ estimate for the averages local density of zeros lying away from the critical axis:
\begin{align*}
\frac{1}{T}\int_0^T \Big| N(\sigma, t + H/\log T) - N(\sigma,t)\Big|\,dt &= \frac{H}{T \log T} \big( N(\sigma, T) - N(\sigma, H/\log T)\big) \\
&\quad+ O\Big(\frac{N(\sigma, T + H/\log T) - N(\sigma, T)}{T}\Big)\\
&\quad + O\Big(\frac{N(\sigma, H/\log T)}{T}\Big) \\
&\ll_{c,H}T^{-c(\sigma-1/2)},
\end{align*}
for any, say, fixed constant $H$.

By a simple bootstrap argument, using this estimate and Fujii's, we show that an $L^k$ estimate may be deduced as well:
\begin{prop}
\label{Lkzero_ex}
For any integer $k \geq 1$, for $1 \leq H \leq T^{1/4}$ and any constant $c \in (0,1/2)$,
\begin{equation}
\label{zd_ex}
\frac{1}{T}\int_0^T \Big| N(\sigma, t + H/\log T) - N(\sigma,t)\Big|^k\,dt \ll_{c,k} H^k T^{-c(\sigma-1/2)}.
\end{equation}
\end{prop}

\textit{Remark:} By modifying our proof, one could increase the range of $c$ to $(0,1)$, and increase the range of $H$ also.

It is by using an estimate of this sort and a few estimates from harmonic analysis that we will remove the Riemann Hypothesis from the proof of Theorem \ref{mainCLT}.

\section{An outline of the proof}
\label{3}

With these preliminaries out of the way, we proceed to outline the proof of Theorem \ref{mainCLT}. It will be convenient, as in a number of works of this sort, to work initially with smoothed averages,
$$
\int \frac{\sigma(t/T)}{T}\cdots\,dt \quad\quad \textrm{instead of} \quad\quad \frac{1}{T}\int_T^{2T} \cdots\,dt,
$$
where $\sigma$ is some positive function of mass $1$ with quadratic decay\footnote{For our technique of proof, something like the quadratic decay of $\sigma$, as opposed to just integrability, will be especially important.} and a compactly supported Fourier transform. 

For a value $T$, we introduce the quantities $A_\gamma = A_\gamma(T) \in \mathbb{R}$, defined by the relation
\begin{equation}
\label{A_gamma}
\frac{1}{2}+i\gamma = \frac{1}{2} + \frac{A_\gamma}{\log T} + i\gamma_0.
\end{equation}
For notational reasons, unless $A_\gamma(\cdot)$ is written explicitly, $A_\gamma$ should always be assumed to be $A_\gamma(T)$.

In Theorem \ref{mainCLT}, equation~\eqref{expectation} is just a consequence of the Riemann-von Mangoldt formula, the demonstration of which we leave to the reader. Equations~\eqref{variance} and \eqref{normal} are deeper and we verify them using the well-known moment method\footnote{See \cite[p. 388]{Bi} or \cite[p. 89]{Ta}, for instance, for introductions to the moment method.}: To prove Theorem \ref{mainCLT}, it is sufficient to demonstrate for $\eta$ and $n(T)$ as in the theorem, for $k = 1,2,3,...$
\begin{multline}
\label{moments}
\int_\mathbb{R}\frac{\mathbf{1}_{[1,2]}(t/T)}{T}\bigg( \sum_\gamma \eta\Big(\frac{\log T}{2\pi n(T)}(\gamma_0-t)\Big) - n(T) \int_\mathbb{R} \eta(\xi)\,d\xi\bigg)^k \, dt \\
= (c_k + o(1))\bigg(\int_{-n(T)}^{n(T)} |x||\hat{\eta}(x)|^2\,dx\bigg)^{k/2}.
\end{multline}
This formula is proved by filling in the details of the steps below. 

\textbf{Step 1:} In the same fashion as \cite{Ro}, one can demonstrate the following computation lemma. We use the notation $K_L(x) = K(x/L)$.

\begin{lemma}
\label{Step1}
Let 
\begin{enumerate}[(i)]
\item $\sigma$ be a fixed non-negative function of mass $1$ with quadratic decay and a compactly supported Fourier transform, 
\item $k$ be fixed positive integer, 
\item $\eta$ be a fixed test function of compact support and bounded variation, 
\item $n(T)$ be a function that tends to infinity as $T\rightarrow\infty$, but such that $n(T) = o(\log T)$.
\item and $K$ be a fixed bump function\footnote{A \emph{bump function} is a function that is smooth and compactly supported. In particular $K$ is at least (in fact much more than) continuously twice differentiable.},  supported in $(-1/8k, 1/8k)$, and with $K(0)=1$.
\end{enumerate}
Then,
\begin{multline}
\label{harmonic_szego}
\int_\mathbb{R} \frac{\sigma(t/T)}{T} \bigg(\lim_{V\rightarrow\infty} \sum_{|\gamma| < V} \check{K}_{n(T)}\ast \eta\Big(\frac{\log T}{2\pi n(T)}(\gamma_0 - i\tfrac{A_\gamma}{\log T} -t)\Big) \\
- \int_{-V}^V \check{K}_{n(T)}\ast \eta\Big(\frac{\log T}{2\pi n(T)}(\xi-t)\Big) \frac{\Omega(\xi)}{2\pi}\,d\xi\bigg)^k\,dt\\
=(c_k +o(1)) \bigg(\int_{-n(T)}^{n(T)} |x||\hat{\eta}(x)|^2\,dx\bigg)^{k/2}.
\end{multline}
\end{lemma}

\textbf{Step 2:} Using a zero density estimate we demonstrate the following:
\begin{lemma}
\label{Step2_1}
Let $\eta, n(T)$, and $K$ be as in Lemma \ref{Step1}. Then
\[
\sum_{\gamma} \Big|\Re \,\check{K}_{n(T)}\ast \eta\Big(\frac{\log T}{2\pi n(T)}(\gamma_0 - i\tfrac{A_\gamma}{\log T} -t)\Big)\Big| < + \infty
\]
for each $t$ and $T$.
\end{lemma}
In particular: both terms in the limit as $V\rightarrow\infty$ of equation \eqref{harmonic_szego} converge, even when taken alone.

Further, using a localized $L^k$ zero density estimate we show that sum over zeros $\gamma_0 + i A_\gamma/\log T$ on the left had side of \eqref{harmonic_szego} is not far from just a sum over the ordinates of zeros, $\gamma_0$.
\begin{lemma}
\label{Step2_2}
Let $\sigma$, $k$, $\eta$, $n(T)$, and $K$ be fixed as in Lemma \ref{Step1}. Then
\begin{align}
\label{harmonic_diff}
\notag \mathcal{E}_1:&= \int_\mathbb{R}\frac{\sigma(t/T)}{T}\bigg( \sum_{\gamma} \Big|\check{K}_{n(T)}\ast \eta\Big(\frac{\log T}{2\pi n(T)}(\gamma_0 - i\tfrac{A_\gamma}{\log T} -t)\Big) - \check{K}_{n(T)}\ast \eta\Big(\frac{\log T}{2\pi n(T)}(\gamma_0 -t)\Big)\Big|\bigg)^k\,dt \\
&\ll 1
\end{align}
\end{lemma} 

\textbf{Step 3:} One may make use of the approach of \cite{Ro} to see that the linear statistics zeros against $\check{K}_{n(T)}\ast \eta$ are not very far from those against $\eta$ itself.

\begin{lemma}
\label{Step3}
Let $\sigma, k, \eta, n(T),$ and $K$ be fixed as in Lemma \ref{Step1}. 
\begin{align}
\label{trunc_diff}
\notag \mathcal{E}_2:=&\int_\mathbb{R} \frac{\sigma(t/T)}{T}\bigg( \sum_\gamma \Big|\eta(\Big(\frac{\log T}{2\pi n(T)}(\gamma_0-t)\Big) - \check{K}_{n(T)}\ast\eta(\Big(\frac{\log T}{2\pi n(T)}(\gamma_0-t)\Big)\Big|\bigg)^k\,dt \\
&\ll 1
\end{align}
\end{lemma}

\textbf{Step 4:} We verify the following computation.
\begin{lemma}
\label{Step4}
Let $\sigma, k, \eta, n(T)$ and $K$ be fixed as in Lemma \eqref{Step1}. Then
\begin{align}
\label{Harmfun_diff}
\notag \mathcal{E}_3 &:= \int_\mathbb{R} \frac{\sigma(t/T)}{T} \bigg| \int_{-\infty}^\infty \check{K}_{n(T)}\ast \eta\Big(\frac{\log T}{2\pi n(T)}(\xi-t)\Big)\frac{\Omega(\xi)}{2\pi}\,d\xi - n(T) \int_{-\infty}^\infty \eta(y)\,dy\bigg|^k \, dt \\
&\ll 1
\end{align}
\end{lemma}

\textbf{Step 5:} Using what we have proved in the above steps, we are able to demonstrate the following:
\begin{lemma}
\label{Step5_1}
Let $\sigma, k, \eta,$ and $n(T)$ be fixed as in Lemma \ref{Step1}. Then
\begin{multline}
\label{smoothed_averages1}
\int_\mathbb{R} \frac{\sigma(t/T)}{T} \bigg( \sum_\gamma \eta\Big(\frac{\log T}{2\pi n(T)}(\gamma_0-t)\Big) - n(T) \int_\mathbb{R} \eta(y)\,dy\bigg)^k \, dt \\
 = (c_k + o(1))\bigg(\int_{-n(T)}^{n(T)} |x||\hat{\eta}(x)|^2\,dx\bigg)^{k/2},
\end{multline}
as long as the integral $\int |x||\hat{\eta}(x)|^2\,dx$ diverges.
\end{lemma}

\begin{corollary}
\label{Step5_2}
Let $k, \eta,$ and $n(T)$ be fixed as in Lemma \ref{Step1}. Then for any non-negative $\sigma$ with quadratic decay and a compactly supported Fourier transform, there exists a $T_0$ such that for $T \geq T_0$,
\begin{multline}
\label{smoothed_averages2}
\int_\mathbb{R} \frac{\sigma(t/T)}{T} \bigg| \sum_\gamma \eta\Big(\frac{\log T}{2\pi n(T)}(\gamma_0-t)\Big) - n(T) \int_\mathbb{R} \eta(y)\,dy\bigg|^k \, dt \\
 \ll \|\sigma\|_{L^1(dt)}\bigg(\int_{-n(T)}^{n(T)} |x||\hat{\eta}(x)|^2\,dx\bigg)^{k/2},
\end{multline}
as long as the integral $\int |x||\hat{\eta}(x)|^2\,dx$ diverges. Here $T_0$ may depend on $\sigma$ (as well as $k, \eta,$ and $n(T)$), but the implied constant has no dependence on $\sigma$.
\end{corollary}

\textbf{Step 6:} The equation \eqref{smoothed_averages1} is very nearly \eqref{moments}, except that we must replace $\sigma$ of the sort delimited above by the function $\mathbf{1}_{[1,2]}$. This is accomplished by a more or less standard argument, approximating $\mathbf{1}_{[1,2]}$ by such $\sigma$.

\section{Upper bounds on counts of zeros}
\label{4}

It will be convenient to define the function
$$
Q(x):= \frac{1}{\pi}\frac{1}{(1+x^2)},
$$
which has mass $1$, and the norms
$$
\|\sigma\|_Q := \pi\,\sup_{x\in\mathbb{R}}\;(1+x^2)|\sigma(x)|.
$$

Theorem \ref{Fuj1} of Fujii serves as an upper bound for us in the following way:
\begin{corollary} 
\label{Fuj2}
For $T \geq 2$,
$$
\int_{\mathbb{R}} \frac{\sigma(t/T)}{T}\Big| N\Big(t+\frac{2\pi(\ell+1)}{\log T}\Big) - N\Big(t + \frac{2\pi \ell}{\log T}\Big)\Big|^k\, dt \ll_k \|\sigma\|_Q,
$$
uniformly for $|\ell| \leq \sqrt{T}$.
\end{corollary}

\begin{proof}
It is easy to see that Theorem \ref{Fuj1} implies for $T\geq 2$,
$$
\int_{(3/2)^n T}^{(3/2)^{n+1} T} \Big| S\Big(t + \frac{2\pi(\ell+1)}{\log T}\Big) - S\Big(t+ \frac{2\pi\ell}{\log T}\Big)\Big|^k\,dt \ll_k (3/2)^n T,
$$
uniformly for $n =0,1,2,...$ and $\ell \leq \sqrt{T}$. (To pass from even $k$, the statement of Theorem \ref{Fuj1} to all $k$, use Cauchy-Schwartz, for instance.) By using the relation \eqref{SofT} between $N(T)$ and $S(T)$, one sees that this implies
$$
\int_{(3/2)^n T}^{(3/2)^{n+1} T} \Big| N\Big(t + \frac{2\pi(\ell+1)}{\log T}\Big) - N\Big(t+ \frac{2\pi\ell}{\log T}\Big)\Big|^k\,dt \ll_k (3/2)^n T,
$$
uniformly for $n$ and $\ell$ in the same range. Moreover, using again Fujii's upper bound, and the pointwise upper bound, Corollary \ref{ptwise} to bound the number of zeros with low height, one sees that
$$
\int_0^T \Big| N\Big(t + \frac{2\pi(\ell+1)}{\log T}\Big) - N\Big(t+ \frac{2\pi\ell}{\log T}\Big)\Big|^k\,dt \ll_k T,
$$
uniformly for $|\ell| \leq \sqrt{T}.$ 

These estimates then imply the corollary, by the vertical symmetry of zeros and the fact that
$$
|\sigma(x)| \leq \|\sigma\|_Q\Big( \mathbf{1}_{[-1,1]}(x) + \sum_{n=0}^\infty \frac{1}{1+[(3/2)^n]^2}\big(\mathbf{1}_{J_n}(x) + \mathbf{1}_{J_n}(-x)\big)\Big),
$$
where for typographical reasons we have defined $J_n = [(3/2)^n, (3/2)^{n+1})$.
\end{proof}

We will ultimately require a slightly more general estimate. We introduce the notation,
\begin{equation}
\label{upper_bound}
M_k\eta(x):= \sum_{\ell \in \mathbb{Z}} \max_{u\in I_\ell(k)}|\eta(u)|\cdot \mathbf{1}_{[\ell,\ell+1)}(x),
\end{equation}
where for notational reasons we denote $I_\ell(k):=[k\ell, k(\ell+1))$, and also the notation,
\begin{equation}
\label{later_bound}
\epsilon_T(\eta):= \sum_{|\ell| > \sqrt{T}} \log(|\ell|+2)\cdot \sup_{u \in [\ell,\ell+1)}|\eta(u)| .
\end{equation}
\begin{corollary}
\label{Fuj3}
For $T \geq 2$,
$$
\int_\mathbb{R} \frac{\sigma(t/T)}{T} \bigg| \sum_\gamma \eta\Big(\frac{\log T}{2\pi}(\gamma_0-t)\Big) \bigg|^k\,dt \ll_k \|\sigma\|_Q \Big(\|M_1 \eta\|^k_{L^1(\mathbb{R})} + (\epsilon_T(\eta) \log T)^k\Big).
$$
\end{corollary}

\textit{Remark:} For $\eta = \mathbf{1}_{[\ell,\ell+1)}$, this is just corollary \ref{Fuj2}.

\textit{Remark:} On the Riemann Hypothesis, it was shown in \cite{Ro} that the right hand side of the above bound may be replaced by
$$
\int_{\mathbb{R}} \frac{\sigma(t/T)}{T} \bigg|\int_{-\infty}^\infty M_1\eta\Big(\frac{\log T}{2\pi}(\xi-t)\Big) \;\log(|\xi|+2)\,d\xi\bigg|^k\,dt.
$$
Likely this bound can be recovered unconditionally, but we do not pursue the matter here.

\begin{proof} Note that
\begin{align*}
\sum_\gamma \eta\Big(\frac{\log T}{2\pi}(\gamma_0-t)\Big) \ll& \sum_{|\ell|\leq \sqrt{T}} \sup_{u\in[\ell,\ell+1)} |\eta(u)| \cdot \Big|N\Big(t + \frac{2\pi(\ell+1)}{\log T}\Big) - N\Big(t + \frac{2\pi\ell}{\log T}\Big)\Big| \\
&+ \sum_{|\ell| > \sqrt{T}} \sup_{u \in [\ell,\ell+1)} |\eta(u)|\cdot \log\big(\big|t + \tfrac{2\pi\ell}{\log T}\big| + 2\big) \\
=:& A + B,
\end{align*}
say. We have used the pointwise upper bound, Corollary \ref{ptwise}, in a very trivial manner to obtain the bound for $|\ell| < \sqrt{T}$.

By H\"{o}lder's inequality and Corollary \ref{Fuj2},
$$
\int_\mathbb{R} \frac{\sigma(t/T)}{T} \prod_{i=1}^k \Big|N\Big(t + \frac{2\pi(\ell_i+1)}{\log T}\Big) - N\Big(t + \frac{2\pi\ell_i}{\log T}\Big)\Big| \ll_k \|\sigma\|_Q,
$$
so it follows from expanding $A^k$ into a multilinear sum,
$$
\int_{\mathbb{R}} \frac{\sigma(t/T)}{T} |A|^k\,dt \ll_k \|\sigma\|_Q \bigg(\sum_{\ell \leq \sqrt{T}} \sup_{u\in[\ell,\ell+1)} |\eta(u)|\bigg)^k \leq \|\sigma\|_Q \|M_1 \eta\|_{L^1(\mathbb{R})}^k.
$$
On the other hand, using the very crude estimate $\log(|t +  2\pi\ell/\log T|+2) \ll \log(|t|+2)\cdot \log(|\ell|+2)$ we know that
$$
B \ll \log(|t|+2)\epsilon_T(\eta).
$$
Therefore,
\begin{align*}
\int_\mathbb{R} \frac{\sigma(t/T)}{T} |B|^k\,dt &\ll \epsilon_T(\eta)^k \|\sigma\|_Q \int_\mathbb{R} \frac{Q(t/T)}{T} \log^k(|t|+2)\,dt \\
& \ll_k \epsilon_T(\eta)^k \|\sigma\|_Q \log^k T.
\end{align*}
As $|A+B|^k \ll_k |A|^k + |B|^k$, we obtain our corollary. 
\end{proof}

We also require upper bouds on the number of zeros that lie away from the critical axis, and the bounds we prove all have their origin in the zero density estimate, Theorem \ref{SelbergJutila}. Note that this estimate implies
\begin{equation}
\label{SelbergJutila2}
\frac{1}{T\log T} \sum_{0< \gamma_0 < T} \mathbf{1}_{(\alpha,\infty)}(|A_\gamma|) \ll_c e^{-c\alpha},
\end{equation}
for any constant $c \in (0,1)$, uniformly for $T \geq 2$, where the notation $A_\gamma$ was defined in equation \eqref{A_gamma}.

Note that, by a crude upper bound,
$$
\sum_\gamma Q\Big(\frac{\gamma_0}{T}\Big) \mathbf{1}_{(\alpha,\infty)}(|A_\gamma|) \leq \sum_{n=1}^\infty Q([n-1]^2) \sum_{0<\gamma_0 < n^2 T} \mathbf{1}_{(\alpha,\infty)}(|A_\gamma|),
$$
and moreover that
$$
|A_\gamma(n^2 T)| = \frac{\log n^2 T}{\log T} |A_\gamma(T)| \geq |A_\gamma(T)|,
$$
so that, using \eqref{SelbergJutila2},
\begin{align*}
\sum_\gamma Q\Big(\frac{\gamma_0}{T}\Big) \mathbf{1}_{(\alpha,\infty)}(|A_\gamma|) &\ll_c e^{-c\alpha} \sum_{n=1}^\infty Q([n-1]^2) n^2 T \log n^2 T \\
&\ll e^{-c\alpha} T\log T.
\end{align*}
Integrating in $\alpha$, we obtain for positive piecewise continuous $f$,
\begin{equation}
\label{SelbergJutila3}
\frac{1}{T\log T} \sum_\gamma Q\Big(\frac{\gamma_0}{T}\Big) f(|A_\gamma|) \ll_c \int_0^\infty f(\xi) e^{-c\xi}\,d\xi.
\end{equation}

By combining this with the Fujii upper bound, Corollary \ref{Fuj3}, we shall obtain,

\begin{lemma}[An $L^k$ zero density estimate]
\label{Lkzerodens}
For $f:\mathbb{R}_+\rightarrow\mathbb{R}$ piecewise continuous and for $T \geq 2$ and $1 \leq H \leq T^{1/4}$,
\begin{equation}
\label{Lkzeroeq}
\int_\mathbb{R} \frac{\sigma(t/T)}{T} \bigg|\sum_\gamma f(|A_\gamma|) Q\Big( \frac{\log T}{2\pi H}(\gamma_0-t)\Big)\bigg|^k \ll_c \|\sigma\|_Q H^k \sqrt{\int_0^\infty f(\xi)^{2k} e^{-c\xi}\,d\xi}.
\end{equation}
for any $c \in (0,1)$.
\end{lemma}

\begin{proof} The left hand side of \eqref{Lkzeroeq} is no more than
\begin{align}
\label{Lkzero_1}
\notag &\|\sigma\|_Q \int_\mathbb{R} \frac{Q(t/T)}{T} \bigg| \sum_\gamma f(|A_\gamma|) Q\Big(\frac{\log T}{2\pi H}(\gamma_0-t) \Big)\bigg|^k\,dt \\
\notag &\leq \|\sigma\|_Q \int_\mathbb{R} \frac{Q(t/T)}{T} \bigg| \sum_\gamma f(|A_\gamma|)^{2k} Q\Big(\frac{\log T}{2\pi H}(\gamma_0-t)\Big)\bigg|^{1/2}\cdot \bigg|\sum_\gamma Q\Big(\frac{\log T}{2\pi H}(\gamma_0-t)\Big)\bigg|^{k-1/2}\,dt \\
\notag &\leq \|\sigma\|_Q \bigg(\int_\mathbb{R}\frac{Q(t/T)}{T}\bigg|\sum_\gamma f(|A_\gamma|)^{2k} Q\Big(\frac{\log T}{2\pi H}(\gamma_0-t)\Big)\bigg|\,dt \bigg)^{1/2} \\
&\hspace{20mm}\times \bigg(\int_\mathbb{R}\frac{Q(t/T)}{T}\bigg|\sum_\gamma Q\Big(\frac{\log T}{2\pi H}(\gamma_0-t)\Big)\bigg|^{2k-1}\,dt\bigg)^{1/2},
\end{align}
where we have used H\"{o}lder's inequality twice.

For $H\leq T^{1/4}$ it is easy to see that
$$
\int_\mathbb{R} \frac{Q(t/T)}{T} Q\Big(\frac{\log T}{2\pi H}(\gamma_0-t)\Big)\,dt \ll \frac{H}{T\log T} Q\Big(\frac{\gamma_0}{T}\Big).
$$
Thus
\begin{align}
\label{Lkzero_2}
\notag \int_\mathbb{R} \frac{Q(t/T)}{T}\bigg| \sum_\gamma f(|A_\gamma|)^{2k} Q\Big(\frac{\log T}{2\pi H} (\gamma_0-t)\Big)\bigg|\,dt &\ll \sum_\gamma \frac{H}{T\log T} f(|A_\gamma|)^{2k} Q\Big(\frac{\gamma_0}{T}\Big) \\
& \ll_c H \int_0^\infty f(\xi)^{2k} e^{-c\xi}\,d\xi.
\end{align}

On the other hand, by Corollary \ref{Fuj3},
\begin{align}
\label{Lkzero_3}
\notag \int_\mathbb{R}\frac{Q(t/T)}{T}\bigg|\sum_\gamma Q\Big(\frac{\log T}{2\pi H}(\gamma_0-t)\Big)\bigg|^{2k-1}\,dt &\ll_k H^{2k-1} + (\epsilon_T(Q_H)\log T)^{2k-1} \\
&\ll_k H^{2k-1},
\end{align}
as $\epsilon_T(Q_H)$ decreases to $0$ for $H\leq T^{1/4}$. (Recall the notation $Q_H(\cdot):=Q(\cdot/H)$.) 

Combining \eqref{Lkzero_1}, \eqref{Lkzero_2}, and \eqref{Lkzero_3} gives the lemma.
\end{proof}

As $\mathbf{1}_{[0,H]} \leq Q_H$, we obtain Proposition \ref{Lkzero_ex} by setting $f = \mathbf{1}_{(\sigma-1/2,\infty)}$.

\textit{Remark:} It is plain that the range $H \leq T^{1/4}$ is not an intrinsic constraint, and that the above argument could be made to work for larger $H$.

\textit{Remark:} The restricted range of $c$ in Proposition \ref{Lkzero_ex} is due to the square root in the term
$$
\sqrt{\int_0^\infty f(\xi)^{2k} e^{-c\xi}\,d\xi.}
$$
By choosing better exponents in our application of H\"{o}lder's inequality, this term may be replaced by
$$
\bigg(\int_0^\infty f(\xi)^{pk} e^{-c\xi}\,d\xi\bigg)^{1/p},
$$
for any $p>1$, and it is in this way that the range of the exponent in Proposition \ref{Lkzero_ex} may be increased. (Though note that this change in procedure also changes the implied constant in the proposition.)

\section{Some upper bounds from harmonic analysis}
\label{5}

\begin{lemma}
\label{pwupper}
For any function $K$, which is even, supported on the interval $[-\kappa, \kappa]$, and continuously twice differentiable, and for any function $\eta$ bounded and with compact support,
$$
\Big| \check{K}_L\ast\eta(x+i\epsilon) + \check{K}_L\ast\eta(x-i\epsilon) - 2 \check{K}_L \ast \eta(x)\Big| \ll_{K,\eta} \frac{\epsilon}{1+x^2}\cdot(1+\epsilon L) e^{2\pi \kappa \epsilon L}.
$$
\end{lemma}

\textit{Remark:} The condition that $\eta$ be bounded with compact support could be relaxed.

\begin{proof}
Note that, by taking the Fourier transform and integration by parts,
\begin{align}
\label{bound}
&\frac{1}{2}\Big(\check{K}_L\ast\eta(x+i\epsilon) + \check{K}_L\ast\eta(x-i\epsilon) - 2 \check{K}_L \ast \eta(x)\Big) \\
\notag &= \int_\mathbb{R} K\Big(\tfrac{\xi}{L}\Big)[\cosh(2\pi \epsilon \xi)-1]e(x\xi) \hat{\eta}(\xi)\,d\xi \\
\notag &= \int_0^\infty \frac{d^2}{d\xi^2}\Big(K\Big(\tfrac{\xi}{L}\Big)[\cosh(2\pi \epsilon \xi)-1]\Big) \cdot \xi \int_\mathbb{R} \big(1-|\omega|/\xi)_+ e(x\omega) \hat{\eta}(\omega)\,d\omega\,d\xi \\
\notag &= \int_0^\infty v_{\epsilon, L}(\xi) \xi^2 \int_\mathbb{R} \Big(\frac{\sin \pi(x-t) \xi}{\pi (x-t)\xi}\Big)^2 \eta(t)\,dt\,d\xi,
\end{align}
where 
\begin{align*}
v_{\epsilon,L}(\xi) :=& \frac{d^2}{d\xi^2}\Big(K\Big(\tfrac{\xi}{L}\Big)[\cosh(2\pi \epsilon \xi)-1]\Big) \\
=& \frac{1}{L^2} K''(\xi/L) [\cosh(2\pi \epsilon \xi)-1] + \frac{4\pi \epsilon}{L} K'(\xi/L) \sinh(2\pi \epsilon \xi) \\
&+ 4\pi^2 \epsilon^2 K(\xi/L) \cosh(2\pi\epsilon \xi).
\end{align*}
Noting that 
$$
\xi^2 \int_\mathbb{R} \Big(\frac{\sin \pi(x-t) \xi}{\pi (x-t)\xi}\Big)^2 \eta(t)\,dt \ll_\eta \frac{1}{1+x^2}
$$
and
$$
\cosh(\xi)-1 \leq |\xi| e^{|\xi|},
$$
we see that \eqref{bound} is bounded by
\begin{align*}
&\ll_{K,\eta} \frac{1}{1+x^2} \int_0^{\kappa L} \frac{\epsilon \xi}{L^2} e^{2\pi \epsilon L} + \frac{\epsilon}{L} e^{2\pi \epsilon \xi} + \epsilon^2 e^{2\pi \epsilon \xi}\,d\xi \\
&\ll \frac{1}{1+x^2}\Big( \epsilon e^{2\pi \kappa L} + \epsilon e^{2\pi \kappa \epsilon L} + \epsilon (\epsilon L) e^{2\pi \kappa \epsilon L}\Big) \\
&\ll \frac{\epsilon}{1+x^2}\cdot (1+\epsilon L)e^{2\pi \kappa \epsilon L}
\end{align*}
as claimed.
\end{proof}

We also cite some results which were proven in \cite{Ro}. The following are Lemmas 5.2 and 5.3 of that paper.

\begin{lemma}
\label{trunc1}
Let $K$ be a fixed bump function with $K(0)=1$, and $f$ be integrable and of bounded variation $\mathrm{var}(f)$. Then,
$$
\|f-\check{K}_L\ast f\|_{L^1(\mathbb{R},dy)} \ll \frac{\mathrm{var}(f)}{L}.
$$
\end{lemma}

\begin{lemma}
\label{trunc2}
Let $K$ be a fixed bump function with $K(0)=1$, and $f(y)\log(|y|+2)$ be integrable with $\int \log(|y|+2)|df(y)|$ bounded. Then,
$$
\|f-\check{K}_L\ast f\|_{L^1(\mathbb{R},\log(|y|+2)dy)} \ll \frac{1}{L}\int_\mathbb{R}\log(|y|+2)|df(y)|,
$$
\end{lemma}

In addition, in section 5 of that paper the following is deduced from these estimates.

\begin{lemma}
\label{trunc_env}
Let $K$ be a fixed bump function with $K(0)=1$, and $f$ be a \emph{fixed} integrable and of bounded variation $\mathrm{var}(f)$. Then,
$$
\| M_{1/L}(\eta - \check{K}_{L}\ast \eta)\|_{L^1(\mathbb{R},dy)} \ll \frac{1}{L}.
$$
\end{lemma}

\section{A proof of the main result}
\label{6}

These steps refer to the outline of our proof in section \ref{3}.

\vspace{2mm}

\textbf{Step 1.} A proof of Lemma \ref{Step1} proceeds in a nearly identical manner as in section 3 of \cite{Ro}. The only difference lies in the fact that here we do not assume the Riemann hypothesis, so the $\gamma$'s that appear in the explicit formula may be complex. Note that for fixed $t$ and $T$, the limit as $V\rightarrow\infty$ of the integrand in \eqref{harmonic_szego} is guaranteed by the explicit formula. 

\vspace{2mm}

\textbf{Step 2.} Let,
\begin{align}
\label{G_gamma}
\notag G_\gamma:=&\bigg( \check{K}_{n(T)}\ast \eta\Big( \frac{\log T}{2\pi n(T)}(\gamma_0 - i \frac{A_\gamma}{\log T} - t)\Big) + \check{K}_{n(T)}\ast \eta\Big( \frac{\log T}{2\pi n(T)}(\gamma_0 + i \frac{A_\gamma}{\log T} - t)\Big) \\
&\hspace{3mm}- 2 \check{K}_{n(T)}\ast \eta\Big(\frac{\log T}{2\pi n(T)}(\gamma_0-t)\Big)\bigg).
\end{align}
\begin{proof}[Proof of Lemma \ref{Step2_1}]
From the harmonic analysis estimate Lemma \ref{pwupper}, with $L= n(T)$, $x = \frac{\log T}{2\pi n(T)}(\gamma_0-t)$, and $\epsilon = \tfrac{A_\gamma}{2\pi n(T)}$,
$$
|G_\gamma| \ll \frac{A_\gamma}{n(T)} (1+A_\gamma/2\pi) e^{A_\gamma/8} Q\Big(\frac{\log T}{2\pi n(T)}(\gamma_0-t)\Big).
$$
For fixed $t$ and $T$, we know that $Q\Big(\frac{\log T}{2\pi n(T)}(\gamma_0-t)\Big)$ decays quadratically in $\gamma_0$. Therefore by the zero density estimate \eqref{SelbergJutila3}, the sum
$$
\sum_\gamma |G_\gamma|
$$
converges absolutely. It is plain that for fixed $t$ and and $T$,
$$
\check{K}_{n(T)}\ast \eta\Big(\frac{\log T}{2\pi n(T)}(\gamma_0-t)\Big)
$$
decays at least quadratically in $\gamma_0$, and therefore the sum of these terms over all zeros converges absolutely as well. Hence from \eqref{G_gamma}, so too does the sum
$$
\sum_\gamma \Big|\check{K}_{n(T)}\ast \eta\Big( \frac{\log T}{2\pi n(T)}(\gamma_0 - i \frac{A_\gamma}{\log T} - t)\Big) + \check{K}_{n(T)}\ast \eta\Big( \frac{\log T}{2\pi n(T)}(\gamma_0 + i \frac{A_\gamma}{\log T} - t)\Big|
$$
converge absolutely.
\end{proof}

\begin{proof}[Proof of Lemma \ref{Step2_2}]
Note that by the symmetry of zeros across the line $\Re(s)=1/2$ we have,
\begin{align*}
\sum_{\gamma} \check{K}_{n(T)}\ast \eta\Big(\frac{\log T}{2\pi n(T)}(\gamma_0 - i\tfrac{A_\gamma}{\log T} -t)\Big) - \check{K}_{n(T)}\ast \eta\Big(\frac{\log T}{2\pi n(T)}(\gamma_0 -t)\Big)=\sum_{\substack{\gamma \\ A_\gamma > 0}} G_\gamma,
\end{align*}
Hence we have,
$$
\mathcal{E}_1 \ll \int_\mathbb{R} \frac{\sigma(t/T)}{T} \bigg( \sum_{\substack{\gamma \\ A_\gamma > 0}} \frac{A_\gamma}{n(T)} (1+A_\gamma/2\pi) e^{A_\gamma/8} Q\Big(\frac{\log T}{2\pi n(T)}(\gamma_0-t)\Big)\bigg)^k\,dt.
$$
By our bootstrapped zero density estimate (with the exponent $c$ set to $1/2$) this is
\begin{align*}
&\ll \|\sigma\|_Q \sqrt{\int_0^\infty \xi^{2k} (1+\xi/2\pi)^{2k} e^{\xi/4} e^{-\xi/2}\,d\xi} \\
&\ll 1
\end{align*}
as claimed.
\end{proof}

\vspace{2mm}

\textbf{Step 3.} 
\begin{proof}[Proof of Lemma \ref{Step3}]
We note that by Corollary \ref{Fuj3} of Fujii's upper bound,
\begin{equation}
\label{E_2_1}
\mathcal{E}_2 \ll \|\sigma\|_Q \big( \|M_1 g_T\|_{L^1(\mathbb{R})}^k + (\epsilon_T(g_T) \log T)^k\big),
\end{equation}
where
$$
g_T(\xi):= \eta\Big(\frac{\xi}{n(T)}\Big) - \check{K}_{n(T)}\ast \eta\Big(\frac{\xi}{n(T)}\Big).
$$

Note that,
$$
M_1(g_T) = \big[M_{1/n(T)}(\eta - \check{K}_{n(T)}\ast \eta)\big]\Big(\frac{\xi}{n(T)}\Big),
$$
so that
\begin{align}
\label{E_2_2}
\notag \|M_1 g_T\|_{L^1(\mathbb{R},dy)} &= n(T) \| M_{1/n(T)}(\eta - \check{K}_{n(T)}\ast \eta)\|_{L^1(\mathbb{R},dy)}\\
&= O(1),
\end{align}
with the bound in the second line following by Lemma \ref{trunc_env}. 

On the other hand, 
\begin{equation}
\label{E_2_3}
\epsilon_T(g_T)\log T = o(1),
\end{equation}
with the rate at which this quantity tends to zero depending upon $\eta$, $K$, and $n(T)$.  For, recalling the definition \eqref{later_bound} of $\epsilon_T$, it suffices to bound both $\epsilon_T(\eta(\cdot/n(T)))$ and $\epsilon_T(\check{K}_{n(T)}\ast \eta(\cdot/n(T)))$. 

The first of these terms clearly vanishes for sufficiently large $T$, since $\eta$ has compact support and $\sqrt{T}$ grows more quickly than $n(T)$.

For the second term, we write
\begin{align*}
  \check{K}_{n(T)} * \eta(\xi / n(T)) = n(T) \int_\mathbb{R} \eta(v) \check{K}(n(T)v-\xi)\,dv.
\end{align*}
Now we recall that $K$ is compactly supported and twice differentiable, so we must have $u^2 \check{K}(u) \to 0$ as $u \to \pm \infty$ by the Riemann-Lebesgue lemma applied to $K''(u)$. Using the fact that $\eta$ is compactly supported, one sees that for $\lvert u \rvert > \sqrt{T}$,
\[
\check{K}_{n(T)} * \eta(u / n(T)) = o( n(T) \lvert u \rvert^{-2} )
\]
and hence
\begin{align*}
  \epsilon_T(g_T) &\ll \sum_{\lvert \ell \rvert > \sqrt{T}} \log(\lvert \ell \rvert + 2) n(T) \lvert \ell \rvert^{-2} \\
  &\ll \frac{n(T) \log T}{\sqrt{T}}
\end{align*}
and \eqref{E_2_3} follows.

Combining \eqref{E_2_1}, \eqref{E_2_2}, and \eqref{E_2_3}, we see that
$$
\mathcal{E}_2 \ll 1,
$$
as claimed.
\end{proof}

\vspace{2mm}

\textbf{Step 4.}
\begin{proof}[Proof of Lemma \ref{Step4}]
Note that
\begin{align*}
&\int_{-\infty}^\infty (\check{K}_{n(T)}\ast \eta - \eta)\Big(\frac{\log T}{2\pi n(T)} (\xi-t)\Big) \frac{\Omega(\xi)}{2\pi} \,d\xi \\
&= n(T) \int_{-\infty}^\infty (\check{K}_{n(T)}\ast \eta - \eta)(y) \frac{\Omega(\tfrac{2\pi n(T)}{\log T}y + t)}{\log T}\,dy \\
&\ll \frac{n(T)}{\log T} \|\check{K}_{n(T)}\ast \eta - \eta\|_{L^1(\log|y|+2)\,dy)} \;+\; \frac{\log(|t|+2)}{\log T} \cdot n(T) \|\check{K}_{n(T)}\ast \eta - \eta\|_{L^1(dy)} \\
&\ll \frac{\log(|t|+2)}{\log T},
\end{align*}
where the first inequality has been deduced from the approximation \eqref{Stirlings} for $\Omega(\xi)$, and the second inequality from Lemmas \ref{trunc1} and \ref{trunc2}. On the other hand, as $\eta$ has compact support,
$$
\int_{-\infty}^\infty \eta\Big(\frac{\log T}{2\pi n(T)}(\xi-t)\Big) \frac{\Omega(\xi)}{2\pi}\,dt = \frac{\log((|t|+2)/2\pi)}{\log T} n(T)\int_{-\infty}^\infty\eta(y)\,dy + O\Big(\frac{1}{|t|+2}\Big).
$$

Hence,
$$
\mathcal{E}_3 = \int_\mathbb{R} \frac{\sigma(t/T)}{T} \bigg| \Big(\frac{\log(|t|+2)-\log T}{\log T}\Big) \cdot n(T) \int_{-\infty}^\infty \eta(y)\,dy + O\Big(\frac{\log(|t|+2)}{\log T}\Big) + O\Big(\frac{1}{|t|+2}\Big)\bigg|^k\,dt.
$$
Because for all $j$,
$$
\int_\mathbb{R} \frac{\sigma(t/T)}{T} \Big(\frac{\log(|t|+2)}{\log T}\Big)^k\,dt = 1 + O_j\Big(\frac{1}{\log T}\Big),
$$
and $n(T) = o(\log T)$, one sees that
\[
\mathcal{E}_3 \ll 1. 
\]
\end{proof}

\vspace{2mm}

\textbf{Step 5.} For notational reasons we adopt the abbreviations
\begin{align*}
\Delta &= \sum_\gamma \eta\Big(\frac{\log T}{2\pi n(T)}(\gamma_0-t)\Big),\\
\Delta' &= \sum_\gamma \check{K}_{n(T)}\ast \eta\Big(\frac{\log T}{2\pi n(T)}(\gamma_0 -t)\Big),\\
\Delta'' &= \sum_{\gamma} \check{K}_{n(T)}\ast \eta\Big(\frac{\log T}{2\pi n(T)}(\gamma_0 - i\tfrac{A_\gamma}{\log T} -t)\Big),\\
\overline{\Delta} &= n(T)\int_{-\infty}^\infty \eta(y)\,dy,\\
\overline{\Delta}' &= \int_{-\infty}^\infty \check{K}_{n(T)}\ast \eta\Big(\frac{\log T}{2\pi n(T)}(\xi-t)\Big)\frac{\Omega(\xi)}{2\pi}\,d\xi.
\end{align*}
\begin{proof}[Proof of Lemma \ref{Step5_1}]
We want to evaluate
$$
\int_\mathbb{R} \frac{\sigma(t/T)}{T}(\Delta-\overline{\Delta})^k\,dt.
$$
Note that
$$
\Delta - \overline{\Delta} = (\Delta'-\overline{\Delta}') + (\Delta' - \Delta'') + (\Delta - \Delta') + (\overline{\Delta}'- \overline{\Delta}).
$$

From Lemma \ref{Step1}
\begin{equation}
\label{abbrev}
\int_\mathbb{R} \frac{\sigma(t/T)}{T}(\Delta'-\overline{\Delta}')^k\,dt = (c_k +o(1)) \bigg(\int_{-n(T)}^{n(T)} |x||\hat{\eta}(x)|^2\,dx\bigg)^{k/2},
\end{equation}
for all $k$. On the other hand, from Lemma \ref{Step2_2}
$$
\int_\mathbb{R} \frac{\sigma(t/T)}{T}|\Delta'-\Delta''|^k\,dt \ll 1,
$$
from Lemma \ref{Step3}
$$
\int_\mathbb{R} \frac{\sigma(t/T)}{T}|\Delta-\Delta'|^k\,dt \ll 1,
$$
and from Lemma \ref{Step4},
$$
\int_\mathbb{R} \frac{\sigma(t/T)}{T}|\overline{\Delta}'-\overline{\Delta}|^k\,dt \ll 1,
$$
for all $k$.\footnote{Note that here the rate of convergence in \eqref{abbrev} and the implied constants elsewhere \emph{do} depend on $k$.}

We have
$$
(\Delta - \overline{\Delta})^k = (\Delta'-\overline{\Delta}')^k + \sum_{\substack{j_1+ j_2 + j_3 + j_4 = k \\ j_1 \leq k-1}} (\Delta'-\overline{\Delta}')^{j_1}(\Delta'-\Delta'')^{j_2}(\Delta-\Delta')^{j_3}(\overline{\Delta}'-\overline{\Delta})^{j_4}.
$$
Using Cauchy-Schwartz and the above bounds, we see that the average of each term of the summand above is
$$
O\bigg(\bigg[\int_{-n(T)}^{n(T)} |x||\hat{\eta}(x)|^2\,dx\bigg]^{(k-1)/2}\bigg).
$$
The average of the remaining term $(\Delta'-\overline{\Delta}')^k$ is obtained from \eqref{abbrev}, and so we obtain the desired result.
\end{proof}

\begin{proof}[Proof of Corollary \ref{Step5_2}]
For $\sigma$ with $\|\sigma\|_{L^1} = 1$, and $k$ even, this corollary is a direct consequence of Lemma \ref{Step5_1}. For $k$ odd, the corollary follows by an application of Cauchy-Schwartz, and then the use of the case that $k$ is even. For general $\sigma$ (with not necessarily unit mass), the result follows by rescaling. 
\end{proof}

\vspace{2mm}

\textbf{Step 6.} We proceed in the same way as \cite{Ro}. The function $S(x) = \Big(\tfrac{\sin \pi x}{\pi x}\Big)^2$ is of quadratic decay and has a Fourier transform with compact support, hence so do all dilations and translations of this function. For any $\epsilon > 0$, it is plain that one may find $\sigma_1$ and $\sigma_2$ which are linear combinations of translations and dilations of $S$, such that 
$$
|\mathbf{1}_{[1,2]}(x) - \sigma_1(x)| \leq \sigma_2(x),
$$
with both $\sigma_1$ and $\sigma_2$ non-negative and $\int \sigma_1 = 1$, while $\int \sigma_2 < \epsilon$. Hence, from Lemma \ref{Step5_1} and Corollary \ref{Step5_2}, and using the abbreviation of Step 5,
\begin{align*}
&\bigg|\int_\mathbb{R}\frac{\mathbf{1}_{[1,2]}(t/T)}{T}(\Delta-\overline{\Delta})^k\,dt - (c_k+o(1))\bigg(\int_{-n(T)}^{n(T)} |x||\hat{\eta}(x)|^2\,dx\bigg)^{k/2}\bigg| \\
&\ll \epsilon \cdot \bigg(\int_{-n(T)}^{n(T)} |x||\hat{\eta}(x)|^2\,dx\bigg)^{k/2},
\end{align*}
for sufficiently large $T$. As $\epsilon$ is arbitrary, we have \eqref{moments} and therefore Theorem \ref{mainCLT}.

\section*{Acknowledgments}

We thank Ashkan Nikeghbali for his encouragement, and the anonymous referee for helpful comments.


\begin{thebibliography}{0}


\bibitem{Be} M. Berry, Semiclassical formula for the number variance of the Riemann zeros, {\it Nonlinearity} {\bf 1} (1988), pp. 399--407.

\bibitem{Bi} P. Billingsley, {\it Probability and Measure, 3rd Edition,} (Wiley-Interscience, 1995).

\bibitem{BoKu} P. Bourgade and J. Kuan, Strong Szeg\H{o} asymptotics and zeros of L-functions, {\it Communications in Pure and Applied Math} {\bf 67}(6) (2014), pp. 1028--1044.

\bibitem{DiEv} P. Diaconis and S. Evans, Linear Functionals of Eigenvalues of Random Matrices, {\it Trans. Amer. Math. Soc.} {\bf 353}(7) (2001), pp. 2615--2633.

\bibitem{Fu2} A. Fujii, Explicit formulas and oscillations, in {\it Emerging Applications of Number Theory}, ed. D. Hejhal, J. Friedman, M. Gutzwiller, A. Odlyzko (Springer, 1999), pp. 219--267.

\bibitem{Fu} A. Fujii, On the zeros of Dirichlet L-functions. I, {\it Trans. Amer. Math. Soc.} {\bf 196} (1974), pp. 225--235.

\bibitem{Fu3} A. Fujii, On the zeros of Dirichlet L-functions. II (with corrections to "On the zeros of Dirichlet $L$-functions. I" and the Subsequent Papers), {\it Trans. Amer. Math. Soc.} {\bf 267} (1981), pp. 33-40.

\bibitem{Gu} A. Guinand, A summation formula in the theory of prime numbers, {\it Proc. London Math. Soc.} {\bf 50} (1948), pp. 107--119.

\bibitem{HuRu} C.P. Hughes and Z. Rudnick, Linear statistics for zeros of Riemann's zeta function, {\it C.R. Acad. Sci. Paris, Ser I.} {\bf 335} (2002), pp. 667--670.

\bibitem{IwKo} H. Iwaniec and E. Kowalski, {\it Analytic Number Theory}, AMS Colloquium Publications, Vol. 53 (AMS, 2004).

\bibitem{Ju} M. Jutila, {\it Zeros of the zeta-function near the critical line,} in ``Studies in Pure
Mathematics to the Memory of Paul Tur\'{a}n", Birkh\"{a}user (1982), pp. 385–-394.

\bibitem{MoVa} H. Montgomery and R.C. Vaughan. {\it Multiplicative Number Theory I. Classical Theory}, Cambridge Studies in Advanced Mathematics, Vol. 97 (Cambridge University Press, 2007).

\bibitem{Ri} B. Riemann, Ueber die Anzahl der Primzahlen unter einer gegebenen Groesse, {\it Monat. der
    Koenigl. Preuss. Akad. der Wissen. zu Berlin aus der Jahre, 1859} (1860), pp. 671--680.

\bibitem{Ro} B. Rodgers, A central limit theorem for the zeroes of the zeta function, {\it Int. J. Number Theory,} {\bf 10} (2)(2014), pp. 483--511. 

\bibitem{Se1} A. Selberg, On the remainder in the formula for $N(T)$, the number of zeroes of $\zeta(s)$ in the strip $0 < t < T$, {\it Avh. Norske Vid. Akad. Oslo. I.} (1) (1944), 27pp.

\bibitem{Se2} A. Selberg, Contributions to the theory of the Riemann zeta-function, {\it Arch. Mat. Naturvid.} {\bf 48}(5) (1946), pp. 89--155.

\bibitem{Si} B. Simon, {\it Orthogonal polynomials on the unit circle, vol. 1}, AMS Colloquium Publications, Vol. 54 (AMS, 2004).

\bibitem{So1} A. Soshnikov, Gaussian Limit for Determinantal Random point Fields, {\it Ann. Probab.} {\bf 30}(1) (2002), pp. 171--187.

\bibitem{So2} A. Soshnikov, The Central Limit Theorem for Local Linear Statistics in Classical Compact Groups and Related Combinatorial Identities. {\it Ann. Probab.} {\bf 28}(3) (2000), pp. 1353--1370.

\bibitem{Sp} H. Spohn. {\it Interacting Brownian particles: a study of Dyson's model,} in ``Hydrodynamic Behavior and Interacting Particle Systems," Springer, IMA Volumes in Mathematics and its Applications, 9 (1987), pp. 151--179.

\bibitem{Sz} G. Szeg\H{o}, On certain Hermitian forms associated with the Fourier series of a positive function, in ``Festschrift Marcel
Riesz", Lund (1952), pp. 222--238.

\bibitem{Ta} T. Tao. {\it Topics in random matrix theory,} Vol. 132. (AMS, 2012).

\bibitem{Ti} E.C. Titchmarsch, rev. by D.R. Heath-Brown. {\it The Theory of the Riemann Zeta-function. 2nd Edition} (Oxford Science Publications/ Clarendon Press, 1986).

\bibitem{We} A. Weil, Sur les ``formules explicites" de la theorie des nombres premiers, {\it Comm. Sem. Math. Univ. Lund [Medd. Lunds Univ. Mat. Sem.]}, Tome Supplementaire. (1952), pp. 252--265.

\end{thebibliography}
\end{document}